\documentclass{amsart}
\usepackage[english]{babel}
\usepackage[dvipsnames]{xcolor}
\usepackage{graphicx,pifont}
\usepackage[utopia]{mathdesign}
\usepackage[a4paper,top=3cm,bottom=3cm,
           left=3cm,right=3cm,marginparwidth=60pt]{geometry}
\usepackage[utf8]{inputenc}
\usepackage{braket,caption,comment,mathtools,stmaryrd}
\usepackage[usestackEOL]{stackengine}
\usepackage{multirow,booktabs,microtype,relsize}
\usepackage{enumerate}
\usepackage{subfigure}

\usepackage[colorlinks,bookmarks]{hyperref} %

      \hypersetup{colorlinks,%
            citecolor=britishracinggreen,%
            filecolor=black,%
            linkcolor=cobalt,%
            urlcolor=cornellred}
      \numberwithin{equation}{section}
\usepackage[capitalise]{cleveref}

\DeclareSymbolFont{usualmathcal}{OMS}{cmsy}{m}{n}
\DeclareSymbolFontAlphabet{\mathcal}{usualmathcal}

\definecolor{cornellred}{rgb}{0.7, 0.11, 0.11}
\definecolor{britishracinggreen}{rgb}{0.0, 0.26, 0.15}
\definecolor{cobalt}{rgb}{0.0, 0.28, 0.67}

\usepackage[colorinlistoftodos]{todonotes} 
\usepackage{amsmath}
\DeclareMathOperator{\rk}{rk}

\newcommand{\Gr}{\mathrm{Gr}}
\newcommand{\sm}{\mathrm{sm}}
\newcommand{\OO}{\mathscr O}
\newcommand{\OZ}{\mathscr Z}
\newcommand{\OH}{\mathscr H}
\newcommand{\OI}{\mathscr I}

\DeclareMathOperator{\Hilb}{Hilb}
\DeclareMathOperator{\Hom}{Hom}
\DeclareMathOperator{\Spec}{Spec}
\DeclareMathOperator{\SET}{SET}
\DeclareMathOperator{\red}{red}
\DeclareMathOperator{\length}{len}
\DeclareMathOperator{\Sch}{Sch}
\DeclareMathOperator{\opp}{op}
\DeclareMathOperator{\Supp}{Supp}

\newcommand{\BA}{{\mathbb{A}}}
\newcommand{\BC}{{\mathbb{C}}}
\newcommand{\BG}{{\mathbb{G}}}
\newcommand{\BN}{{\mathbb{N}}}
\newcommand{\BZ}{{\mathbb{Z}}}

\newcommand{\Fm}{{\mathfrak{m}}}

\newcommand{\hilbert}[2]{{\Hilb^{#1}{(#2)}}}

\usepackage{tikz}
\usepackage{tikz-cd}
\usepackage{comment}
\usetikzlibrary{matrix,shapes,intersections,arrows,decorations.pathmorphing}
\tikzset{commutative diagrams/arrow style=math font}

\theoremstyle{definition}

\newtheorem{definition}{Definition}[section]
\newtheorem{example}[definition]{Example}
\newtheorem{conjecture}[definition]{Conjecture}
\newtheorem{convention}[definition]{Convention}
\newtheorem{notation}[definition]{Notation}
\newtheorem{remark}[definition]{Remark}

\newtheoremstyle{thm}
        {3mm}
        {3mm}
        {\slshape}
        {0mm}
        {\bfseries}
        {.}
        {1mm}
        {}
\theoremstyle{thm}
\newtheorem{theorem}[definition]{Theorem}
\newtheorem{lemma}[definition]{Lemma}
\newtheorem{prop}[definition]{Proposition}

\newtheorem{thm}{Theorem}

\crefname{notation}{Notation}{Notations} 
\crefname{thm}{theorem}{theorems}
\crefname{conjecture}{conjecture}{conjectures}
\crefname{prop}{proposition}{propositions}

\usepackage[normalem]{ulem}

\usetikzlibrary {patterns,patterns.meta}

\newcommand{\bh}{{\boldsymbol{h}}}

\title[Components of the nested Hilbert scheme of few points]{Components of the nested Hilbert scheme of few points} 

\author[M.~Graffeo]{Michele Graffeo} \address[M.~Graffeo]{SISSA\\ Via Bonomea 265\\ 34136\\ Trieste \\Italy  \\and INFN Sezione Trieste\\ Politecnico di Milano\\ Piazza Leonardo da Vinci 32\\ 20133 Milan\\ Italy}
\email{mgraffeo@sissa.it}

\author[P.~Lella]{Paolo Lella} \address[P.~Lella]{Dipartimento di Matematica\\ Politecnico di Milano\\ Piazza Leonardo da Vinci 32\\ 20133 Milan\\ Italy}
\email{paolo.lella@polimi.it}

\begin{document}

\begin{abstract}
We study the existence and the schematic structure of elementary components of the nested Hilbert scheme on a smooth quasi-projective variety. More precisely, we establish a new lower bound for the existence of non-smoothable nestings of fat points on a smooth $n$-fold, for $n\geqslant 4$. Moreover, we develop a systematic method for building generically non-reduced elementary components. We also investigate the irreducibility of the Hilbert scheme of points on a singular hypersurface of $\mathbb A^3$. Explicitly, we show that the Hilbert scheme of points on a hypersurface of $\mathbb{A}^3$ with a singularity of multiplicity at least 5 admits elementary components. 
\end{abstract}

\maketitle

\section{Introduction}

The nested Hilbert scheme, which we simply call the Hilbert scheme throughout the paper, of points on a quasi-projective variety $X$ is the  scheme locally of finite type $\Hilb (X)$, representing the functor
\begin{equation}\label{eq:nestintro}
\begin{array}{rcccl}
\Sch_{\BC}^{\opp} &&\xrightarrow{\underline{\Hilb}{(X)}} && \SET \\
B && \longmapsto && \left\{\OZ^{(1)}\subset\cdots \subset \OZ^{(r)}\subset X\times B\ \middle\vert\ \begin{array}{l} r\in\BZ_{>0},\  \OZ^{(i)}\textnormal{~ closed,}\\\textnormal{$B$-flat,~$B$-finite}\end{array}\right\}.
\end{array} 
\end{equation}
Its existence was proven by Grothendieck in \cite{Grothendieck_Quot}, for $r=1$, and by Kleppe in general, see \cite{Kleppe}. By results of Hartshorne \cite{HartshorneConnect}, Fogarty \cite{FOGARTY} and Cheah \cite{CHEACELLULAR}, the decomposition of   $\Hilb (X)$ into  connected components is 
{ 
\begin{equation}\label{eq:cc}
    \Hilb (X)= \coprod_{r\geqslant 1}\coprod_{(d_1\leqslant\ldots\leqslant d_r)\in\BZ^r_{>0}} \Hilb^{(d_1,\ldots,d_r)} (X),
\end{equation}}
where,  given a {non-decreasing}  sequence $\underline{d}=(d_1,\ldots,d_r)\in\BZ^r_{>0}$, $\Hilb^{\underline{d}} (X)$ is the  locus whose $B$-points correspond to nestings of the form \eqref{eq:nestintro} with $\length_B\OZ^{(i)}=d_i$, for $i=1,\ldots,r$.

In this paper, we investigate two aspects of the geometry of $\Hilb (X)$ when $X$ is smooth, as well as one aspect of the singular setting. Specifically, when $X$ is smooth,  we present an infinite family of non-smoothable nestings of fat points, which we interpret as \lq\lq minimal\rq\rq, and we provide a method for constructing generically non-reduced elementary components starting from a generically reduced one. On the other hand, we study the reducibility of $\Hilb^d (X)$ for $d\in \BZ_{>0}$ and $X$ a singular surface.

An irreducible component of $\Hilb (X)$ is elementary if its closed points correspond to nestings of irreducible schemes, i.e.,~spectra of local Artinian $\BC$-algebras of finite type. We call such schemes fat points. Elementary components are considered the building blocks of the Hilbert scheme of points, since every other component is generically étale-locally a product of elementary ones. For this reason, they play a central role in the study of the geometry of Hilbert schemes of points. The case $r=1$ is already interesting, especially for $\dim X = 3$. Indeed, in this setting $\Hilb^{d_1}(X)$ is known to be irreducible for $d_1\leqslant 11$ (see \cite{10points,JOACHIM}) and reducible for $d_1\geqslant 78$ (see \cite{IARRO}). Very recently, Michałek presented AI-assisted proofs of the irreducibility of $\Hilb^{d_1}(X)$ also in the range $12 \leqslant d_1 \leqslant 21$ \cite{Michalek14,Michalek19,Michalek20,Michalek21}.
Nothing is known in the range $21 < d_1 < 78$ and no elementary component of $\Hilb^{d_1}(X)$ is known. On the other hand, it is well known that when $\dim X\geqslant 4$, the Hilbert scheme $\Hilb^{d_1}(X)$ is irreducible if and only if $d_1\leqslant 7$ (see \cite{Iarrob,MAZZOLA}). Despite this, only some sporadic examples and a few infinite families of elementary components are known; see \cite{ERMANVELASCO,UPDATES,2step,SomeElementary,MoreElementary,ELEMENTARY,Satrianostaal,GALOIS,Sha90}.

As the length of the nesting increases, the question about irreducibility is already interesting when $X$ is a smooth surface. For every pair of integers $0<d_1\leqslant d_2$, $\Hilb^{d_1,d_2}(X)$ is irreducible by the results in \cite{Rasul-irr-nested}. The first example of a reducible nested Hilbert scheme on a surface was given in \cite{consiglio}, where reducibility was established for $r \geqslant 23$. Later, it was proven in \cite{ALESSIONESTED} that, for every $r\geqslant 5$, there exist non-decreasing sequences $\underline{d}\in\BZ_{>0}^r$ such that $\Hilb^{\underline{d}}(X)$ is reducible. More recently, this lower bound was improved in \cite{2step}, where $\Hilb^{454, 491, 527, 565}(X)$ was shown to be reducible. The question of irreducibility for $r=3$ remains completely open.

For a smooth quasi-projective variety $X$, the problems considered in this paper are local in nature. Therefore, it is enough to work in the case $X=\mathbb{A}^n$, which simplifies the notation and allows us to describe the relevant objects in terms of ideals in a polynomial ring in $n$ variables.

\smallskip 

\subsubsection*{Short nestings} The search for elementary components has attracted considerable attention over the years. However, the known results are scattered and only a few papers present explicit infinite families. 
In particular, we are interested in those elementary components of $\Hilb(\BA^n)$ parametrising schemes of maximal embedding dimension, i.e., components that do not arise from the Hilbert scheme of points of a smaller affine space. We refer to these components as primitive elementary components, see \Cref{def:componets}. 
It was shown in \cite{ERMANVELASCO} that, for $r=1$, the smallest possible value of $d_1\in\BZ_{>0}$ for which there exists a primitive component of $\Hilb^{d_1}(X)$ is  $d_1=\dim X+4$.
Our first main result is that, by allowing nestings, this phenomenon occurs sooner. 
 
\begin{thm}\label{THMINTRO:A}
   Let $R$ be the polynomial ring in $n\geqslant 4$ variables and let $2\leqslant s \leqslant n-2$ be an integer. Then, the scheme $ \Hilb^{1+s,n+3}(\BA^n)$ has a primitive elementary component.  
\end{thm} 

It is worth mentioning that, in all the cases we have checked, the primitive components mentioned  in \Cref{THMINTRO:A} are generically reduced. We conjecture that this is true in general in \Cref{conj:TNT}, and we confirm it for $n\leqslant 15$  in \Cref{prop:checked}.

\smallskip

\subsubsection*{Non-reduced elementary components} The second question we address concerns the schematic structure of elementary components on a smooth quasi-projective variety. Our interest stems from the fact that although the TNT property is a very powerful tool for detecting elementary components, it is only suitable for reduced ones, and no systematic method has been developed for dealing with generically singular components, as discussed in \cite{UPDATES,pathologies,jelisiejewnonred,oszer,13PUNTI}. For this reason, the generically non-reduced elementary components of $\Hilb(\BA^n)$ are almost invisible at first glance, and more effort is required to find them. In our second main result, we give a recipe for building a generically non-reduced elementary component starting from a reduced one.

\begin{thm}\label{THMINTRO:B}
     Let $V\subset \Hilb^{\underline{d}} (\BA^n)$ be a generically reduced  elementary component. Assume that there is at least one point $[({I}^{(i)})_{i=1}^r]\in V$ and some $k\in\BZ_{>0}$ such that $I^{(j)}\supsetneq \Fm^k\supsetneq I^{(j+1)} $, for some $1\leqslant j \leqslant r-1$. Put 
    \[
    \widetilde{\underline {d}}=\left(d_1,\ldots, d_j,\binom{n+k-1}{n},d_{j+1},\ldots,d_r\right).
    \]  
    Then, there is a generically non-reduced elementary component $\widetilde V\subset \Hilb^{\widetilde{\underline{d}}}(\BA^n)$ such that  
$ \widetilde V_{\red}\cong   V_{\red} . $
\end{thm}
 
We remark that \Cref{THMINTRO:B} is a direct generalisation of  \cite[Theorem 5]{UPDATES}, as explained in \Cref{rem:generalisationold}. The novelty here is that we are able to decompose the negative tangent spaces at the points under consideration into direct sums of more tractable summands. In particular, in the homogeneous case, we give a much more explicit decomposition in \Cref{cor:thmB-2step-homogeneous}.  

\smallskip

\subsubsection*{The Hilbert scheme of singular surfaces} The third and last aspect of the geometry of the Hilbert scheme  of points we study is the irreducibility of $\Hilb ^d (X)$, for $d\in\BZ_{>0}$, and $X$ a singular variety. Since the scheme $\Hilb^d(X)$ is reducible for $d\gg 0$ when $X$ is smooth and $\dim X \geqslant 3$, the question is most interesting in dimension 1 and 2. The curve case has been extensively studied in the literature, see \cite{AA2,AA3,AA1,AA5,luan2023irreduciblecomponentshilbertscheme,AA4,ágoston2023analyticlatticecohomologyisolated} and references therein.  In this setting, if $C$ is a curve, then $\Hilb^2(C)$ is irreducible if and only if $C$ has only planar singularities. In \cite{Angel}, the authors provided a combinatorial description of the geometry of $\Hilb(C)$ for curves $C$ with only transverse singularities. More precisely, the authors prove that $\Hilb(C)$ is reduced and give an explicit description of its irreducible components and their normalisations. In the present work, we focus on the surface case. Although it is possible to build many examples of reducible Hilbert schemes when $S$ has singularities of embedding dimension strictly greater than 3, see \Cref{ex:RNC}, the same is not true for hypersurfaces $S\subset \BA^3$. It was proven independently  in \cite{soren} and \cite{Zheng} that, if $S$ has only ADE singularities, $\Hilb^d(S)$ is irreducible for all $d> 0$. However, as we show in \Cref{ex:iarrosurf}, as a consequence of the results in the celebrated paper \cite{IARRO} by Iarrobino, if $S$ has a singularity of multiplicity at least 8, the scheme $\Hilb^{78}(S)$ is reducible. Building upon the theory of 2-step ideals developed in \cite{2step} we prove our main result in this direction.

\begin{thm}\label{THMINTRO:C}
Let $S\subset \BA^3$ be a hypersurface. If $S$ has a singular point of multiplicity at least 5, then $\Hilb^d (S) $ is reducible for $d\geqslant 22$. 
\end{thm}

\begin{remark}
    All singularities of type ADE have multiplicity 2. Therefore, the problem of characterising the singular hypersurfaces $S\subset\BA^3$ for which $\Hilb^d(S)$ is reducible, for $d$ large enough, remains open for hypersurface singularities of multiplicity 4, 3, and certain singularities of multiplicity 2. A natural question that arises is whether this could be another way of characterising ADE singularities, see \cite{zbMATH03650712}. Note that this question is closely related to  \textbf{Problem XII} in the list of open problems proposed in \cite{JJ-Hilb-open-problems}.
\end{remark}

\subsection{Organisation of the content}
In \Cref{sec2} we introduce the basic tools needed to prove our main results. Precisely, in 
\Cref{subsec:BB} we introduce the  Bia{\l{}}ynicki-Birula decomposition and relate it in  \Cref{subsec:HS} to that of Hilbert--Samuel stratification. Then, in \Cref{subsec:TNT} we recall the definition and the main consequences of the TNT property, and in \Cref{subsec 2-step} we recall the definition of 2-step ideals without linear syzygies and their number of moduli.  Finally, in Sections \ref{sec3}, \ref{sec4} and \ref{sec5} we prove Theorems \ref{THMINTRO:A}, \ref{THMINTRO:B} and \ref{THMINTRO:C} respectively.

\subsection*{Acknowledgments} We thank Franco Giovenzana, Luca Giovenzana, Alessio Sammartano and   Klemen \v{S}ivic for interesting conversations.  The second author was supported by the grant PRIN 2022K48YYP {\em Unirationality, Hilbert schemes, and singularities}. Both authors are members of  GNSAGA - INdAM

\section{Preliminaries} \label{sec2}

In this section, we introduce the main tools that we need to prove Theorems \ref{THMINTRO:A}, \ref{THMINTRO:B} and \ref{THMINTRO:C} from the introduction. The central topics are the Bia{\l{}}ynicki-Birula decomposition in \Cref{subsec:BB}, the Hilbert--Samuel stratification in \Cref{subsec:HS}, the TNT property in \Cref{subsec:TNT} and 2-step ideals without linear syzygies in \Cref{subsec 2-step}. 

\begin{convention} 
        To simplify the notation, for any vector $\underline{d}\in \BZ^r$ we denote by $d_i$, for $i=1,\ldots,r$, its entries. Moreover, if $\underline{d}\in \BZ^r_{>0}$ is a non-decreasing sequence of positive integers, by $\underline{d}$-nesting (or simply $r$-nesting) $\underline{Z}$ in $X$ we mean that $\underline{Z}=(Z^{(1)},\ldots,Z^{(r)})$, where $Z^{(1)}\subset\cdots \subset Z^{(r)}\subset X$ are closed zero-dimensional subschemes and $\length Z^{(i)}=d_i$, for $i=1,\ldots,r$. Finally, the support of the nesting $\underline Z$ is the set-theoretic support of the scheme $Z^{(r)}$, i.e.,~$\Supp \underline{Z}= \Spec( \OO_X /\sqrt{\OI_{Z^{(r)}}})$.

         Since the questions we address in \Cref{sec3,sec4} are local in nature and the results therein concern  smooth quasi-projective varieties, it is not restrictive to put $X\cong \BA^n$ and hence to work up to étale covers. In view of the fact that the Hilbert scheme represents the Hilbert functor, we denote points of the $\underline{d}$-nested Hilbert scheme $\hilbert{\underline{d}}{\BA^n}$ by $[Z^{(1)}\subset \cdots \subset Z^{(r)}]$ or $[I_{Z^{(1)}}\supset \cdots \supset I_{Z^{(r)}}]$ referring to both as $\underline{d}$-nestings (or simply $r$-nestings).

         Finally, we denote by $R=\BC[x_1,\ldots,x_n]$ the polynomial ring in $n$ variables with complex coefficients, and by $\Fm=(x_1,\ldots,x_n)\subset R$ the maximal ideal generated by the variables. 
\end{convention}

\subsection{The Bia{\l{}}ynicki-Birula decomposition}\label{subsec:BB}
 
Consider  the diagonal action on $ \hilbert{\underline{d}}{\BA^n} $ of the torus $\mathbb G_m=\Spec \BC[s,s^{-1}]$  induced by homotheties of $\BA^n$. Then, the Bia{\l{}}ynicki-Birula decomposition is the  scheme locally of finite type $\hilbert{+}{\BA^n}$ representing the following functor
\[ 
   \left( \underline{\Hilb}^{+}(\BA^n)\right)(B) = \Set{\varphi\colon \overline{\mathbb G}_m \times B \rightarrow \hilbert{}{\BA^n} |   \varphi \mbox{ is $\mathbb G_m$-equivariant}} ,
\]
where by convention $\overline{\mathbb G}_m= \Spec \BC[ s^{-1}]$. Every $B$-point of $\Hilb^+(\BA^n)$ corresponds to a $(\overline{\BG}_m\times B)$-point of $\Hilb(\BA^n)$, and this natural transformation of functors induces an injective morphism $\Hilb^+(\BA^n)\hookrightarrow\Hilb(\BA^n)$. Set-theoretically its image agrees with  the subset of the nested Hilbert scheme parametrising fat nestings supported at the origin $0\in\BA^n$, i.e.,~nestings of $\Fm$-primary ideals. In what follows, we make extensive use of this correspondence by identifying points of $\Hilb^+(\BA^n)$ with points of $\Hilb(\BA^n)$.  
According to the notation adopted above, we put
\[
\Hilb^+(\BA^n)=\coprod_{r\geqslant1}\coprod_{\underline d\in\BZ^r_{>0}} \Hilb^{\underline d,+} (\BA^n),
\]
where $\Hilb^{\underline d,+} (\BA^n)=\Hilb^{\underline d}( \BA^n)\cap \Hilb^{+} (\BA^n)$ if $\underline{d}$ is non-decreasing, and empty otherwise.

\begin{remark}
The Bia{\l{}}ynicki-Birula decomposition has proven to be a powerful tool in the study of the geometry of Hilbert and Quot schemes of points. For instance, using this technique, the authors of \cite{nestedsurf} compute the number of irreducible components of maximal dimension for an infinite family of punctual nested Hilbert schemes on a smooth surface. Similar results are obtained in \cite{MOTIVES} for punctual Hilbert and Quot schemes of a small number of points on a smooth $n$-fold, and in \cite{indscheme} for Quot and  nested Hilbert schemes on the ind-scheme $\mathbb A^\infty$. 
\end{remark}

Consider an $\Fm$-primary ideal $I\subset R$, and put 
\[ (I)_{\geqslant k} = I\cap \Fm^k\qquad\mbox{ and }\qquad (R/I)_{\geqslant k}= (\Fm^k + I)/I \subset R/I. \]

\begin{definition}\label{def:negativetangents}
Let $[\underline I] \in\Hilb^+(\BA^n)$ be a point.
Then, the \textit{non-negative part of the tangent space} $\mathsf T_{ [\underline{I}]}  \hilbert{}{\BA^n}$  is the following vector subspace
\[
\mathsf T_{ [\underline{I}]}^{\geqslant0}  \hilbert{}{\BA^n} = \left\{\varphi\in\mathsf T_{ [\underline{I}]} \hilbert{}{\BA^n}\ \middle\vert\ \varphi\big((I^{{(i)}})_{\geqslant k}\big) \subset (R/I^{{(i)}})_{\geqslant k}\mbox{ for all }k \in\BN\mbox{ and for }i=1,\ldots,r\right\}. 
\]
The \textit{negative tangent space} at $[\underline{I}]\in \hilbert{}{\BA^n}$ is
\[
\mathsf T_{ [\underline{I}]}^{< 0}  \hilbert{}{\BA^n}=\frac{ \mathsf T_{ [\underline{I}]}  \hilbert{}{\BA^n}}{ \mathsf T_{ [\underline{I}]}^{\geqslant0}  \hilbert{}{\BA^n}}.
\]
\end{definition}
The following proposition from \cite{ELEMENTARY} expresses the tangent space   $\mathsf T_{ [\underline{I}]}  \Hilb^{+} (\BA^n)$  in  terms of the non-negative tangent space at $[\underline{I}]\in \hilbert{}{\BA^n}$. 
\begin{prop}[{\cite[Theorem 4.11]{ELEMENTARY}}]\label{rem:nonneg-punctual}
Let $[\underline{I}]\in \hilbert{+}{\BA^n}$ be a fat nesting. Then, we have
\[
    \mathsf T_{ [\underline{I}]}  \hilbert{+}{\BA^n}= \mathsf T_{ [\underline{I}]}^{\geqslant0}  \hilbert{}{\BA^n} .
\]    
\end{prop} 
\begin{notation}
    In the rest of the paper, we denote by $\mathsf t _{ [\underline{I}]}   \hilbert{}{\BA^n}$ the complex dimension
    \[
    \mathsf t _{ [\underline{I}]}   \hilbert{}{\BA^n}=\dim_{\BC}\mathsf T_{ [\underline{I}]}   \hilbert{}{\BA^n},
    \]
    and similarly for the graded pieces of the tangent space. Moreover, if $M,N$ are $R$-modules and $\Hom_R(M,N)$ is finite dimensional over $\BC$, we denote by $\hom_R(M,N)$ its complex dimension.
\end{notation}
\subsection{Hilbert functions and embedding dimension} \label{subsec:HS}

\begin{definition}
Let $A=\bigoplus_{k\in\BZ} A_k$ be a graded $\BC$-algebra of finite type and let $M=\bigoplus_{k\in\BZ} M_k$ be a finitely generated graded $A$-module. The \emph{Hilbert function} $\bh_M$ associated to $M$ is the function $\bh_M : \BZ \to \BN$ defined by
\[
k\   \mapsto\ \dim_{\BC} M_k.
\]
 Let $(A,\Fm_A)$ be a  local Artinian $\BC$-algebra of finite type. The \emph{Hilbert function} $\bh_A$ of $A$ is defined to be the Hilbert function of its associated graded algebra  
\begin{equation}\label{eqn:grading-A-module}
\mathsf{gr}_{\Fm_A}(A) =\bigoplus_{k\geqslant 0}\, \mathfrak m_A^{k}/\mathfrak m_A^{k+1},
\end{equation}
where $\mathsf{gr}_{\Fm_A}(A)$ is seen as a graded module over itself.  
\end{definition}
\begin{notation}
 For the ideal $I\subset R$ of a fat point supported at the origin, the function $\bh_{R/I}$ eventually vanishes, therefore we represent it as a tuple of positive integers. { Moreover, if $\underline{I}$ is an $r$-nesting, then $\underline{\bh}_{R/\underline{I}}=\left({\bh}_{R/{I^{(i)}}}\right)_{i=1}^r$ denotes the sequence of Hilbert functions of the quotient algebras corresponding to the given nesting.} 
\end{notation} 

Recall that the map $[\underline{I}] \mapsto \underline{\bh}_{R/\underline{I}}$
is locally constant on $\Hilb ^{+}(\BA^n)$, see \cite[Prop.~3.1]{ELEMENTARY}. This induces, for every $\underline{\bh}=(\bh_1,\ldots,\bh_r):\BZ\to\BN^r$, a canonical scheme structure of finite type on the (possibly empty) locally closed subsets
\begin{equation}\label{eq:defHSStratum}
    H_{\underline{\bh}}^n=\Set{ \left[\underline{I}\right] \in \Hilb^+  \BA^n | \underline{\bh}_{R/\underline{I}}\equiv \underline{\bh} }\subset \Hilb ^{ (|\bh_i|)_{i=1}^r,+}(\BA^n)\subset \Hilb ^+(\BA^n),
\end{equation} 
     where $|\bh_i|=\sum_{k\geqslant 0} \bh_i(k)$ denotes the size of $\bh_i$, for $i =1,\ldots,r$.      
\begin{definition}
Given a $r$-tuple  $\underline{\bh}=(\bh_1,\ldots,\bh_r):\BZ\to\BN^r$ of   functions, the \emph{Hilbert--Samuel stratum} $ H_{\underline{\bh}}^n$  is the (possibly empty) locally closed subset given in \eqref{eq:defHSStratum}, endowed with the schematic structure induced by the Bia{\l{}}ynicki-Birula decomposition.
\end{definition}

Recall that there is a surjective morphism of schemes locally of finite type
    \[
         \Hilb ^{+}(\BA^n) \xrightarrow{\ \pi^n\ } \big(\Hilb(\BA^n)\big)^{\mathbb G_m},
    \] 
    corresponding to sending a $B$-point $[\varphi:B\times \overline\BG_m\rightarrow \Hilb(\BA^n)]$ to its limit $[\varphi|_{B\times \Set{\infty}}:B\times \Set{\infty}\rightarrow \Hilb(\BA^n)]$ which is a $B$-point of $\Hilb(\BA^n)$. Denote by $\OH_{\underline{\bh}}^n$ the $\BG_m$-fixed part of the ($\BG_m$-invariant) open subset $H_{\underline{\bh}}^n\subset \Hilb^+(\BA^n)$, i.e.,~$\OH_{\underline{\bh}}^n=(H_{\underline{\bh}}^n)^{\BG_m}$. Then, the morphism $\pi^n$ restricts to a morphism $\pi_{\underline{\bh}}^n:  H_{\underline{\bh}}^n \to \OH_{\underline{\bh}}^n$.

    Moreover, given a closed point $[\underline I]\in(\Hilb(\BA^n))^{\BG_m}\subset \Hilb(\BA^n)$, we have the following isomorphisms (see \cite{multigraded})
    \[
    \mathsf{T}_{[\underline{I}]} (\pi^n)^{-1}([\underline I]) \cong  \mathsf T_{ [\underline{I}]}^{>0}  \Hilb  (\BA^n),\quad\mbox{ and }\quad\mathsf T_{ [\underline{I}]}  (\Hilb(\BA^n))^{\mathbb G_m} \cong \mathsf T_{ [\underline{I}]}^{=0}  \Hilb  (\BA^n).
    \]

\begin{remark}\label{rem:BBHOMOo} Set-theoretically, the fixed locus of  the diagonal action of the torus $\mathbb G_m=\Spec \BC[s,s^{-1}]$ on $ \Hilb^{} (\BA^n) $ agrees with the locus parametrising nestings of homogeneous ideals. Given a nesting $ [\underline{I}]\in(\Hilb(\BA^n)) ^{\BG_m}$, the  $\mathbb G_m$-action lifts to the tangent space $\mathsf T_{[\underline{I}]} \Hilb^{}(\BA^n)$ and induces a weighted direct sum decomposition  
    \[
    \mathsf T_{[\underline{I}]} \Hilb^{ }(\BA^n)=\bigoplus_{k\in \BZ } \mathsf T_{[\underline{I}]}^{=k} \Hilb^{}(\BA^n).
    \]
    This direct sum decomposition is consistent with \Cref{def:negativetangents} meaning that
    \[
    \mathsf T_{[\underline{I}]}^{\geqslant 0}\Hilb(\BA^n)=\bigoplus_{k\geqslant 0 } \mathsf T_{[\underline{I}]}^{=k} \Hilb(\BA^n)\quad \mbox{ and }\quad   \mathsf T_{[\underline{I}]}^{<0} \Hilb(\BA^n) \cong \bigoplus_{k<0 } \mathsf T_{[\underline{I}]}^{=k} \Hilb(\BA^n),
    \]
    see \cite[Section 2]{ELEMENTARY} for more details.
\end{remark}

We conclude this subsection by reporting some basic properties of the Hilbert--Samuel strata in the case $r=1$ that we shall use in the next sections.

\begin{lemma}[{\cite[Prop.~3.1]{8POINTS}}, {\cite[Lemma 2.9 and Theorem 3.4]{MOTIVES}}]\label{lemma:tech2}
Fix $\bh = (1, h_1, \ldots , h_t)$. Then, for every $n\geqslant h_1$, there is a Zariski locally trivial fibration with fibre $\OH_{\bh}^{h_1}$   
\begin{equation}\label{eqn:rho-map}
\begin{tikzcd}
\OH_{\bh}^n\arrow[r,"\rho^n_{\bh}"]&\Gr(n-h_1,R_1),
\end{tikzcd}
\end{equation}
which on closed points sends a homogeneous ideal $[I]\in\OH_{\bh}^n$ to its linear part $[I_1]\in \Gr(n-h_1,R_1)$ . Moreover, the composite morphism
\[
\begin{tikzcd}
  H^n_{\bh} \arrow{r}{\pi^n_{\bh}} & \OH^n_{\bh} \arrow{r}{\rho^n_{\bh}} & \Gr(n-h_1,R_1)
\end{tikzcd}
\]
is Zariski locally trivial. Finally, there is a Zariski locally trivial fibration
\[
\begin{tikzcd}
    H_{\bh}^n\arrow[r]& H_{\bh}^{h_1},
\end{tikzcd}
\]
with fibre an $\BA^{(n-h_1)(d-h_1-1)}$-bundle over $\Gr(h_1,n)$. 
\end{lemma}

\subsection{Trivial Negative tangents and elementary components} \label{subsec:TNT}

	Recall that the $\underline{d}$-nested Hilbert scheme of points on an irreducible quasi-projective variety $X$  always has  a distinguished irreducible component of dimension $d_r\cdot \dim X$. Precisely, the \emph{smoothable component}
    \[
    \Hilb^{\underline{d}}_{\mathrm{sm}}{( X)}=\overline{\Set{[\underline{Z}]\in\Hilb(X) | Z^{(r)} \mbox{ is reduced}}}\subset \Hilb(X).
    \]
    We refer to schemes corresponding to its closed points as smoothable schemes. 
	
	\begin{definition}\label{def:componets} An irreducible component $V\subset   \hilbert{}{X}$ is \textit{elementary} if it parametrises just fat nestings. An elementary component $V\subset   \hilbert{}{X}$ is \textit{primitive} if the generic point $[\underline{Z}]\in V$ has $\bh_{\OO_{Z_r}}(1)=n $.   
	\end{definition}

Elementary components are considered the building blocks of the Hilbert schemes of points as each irreducible component is shown in \cite{Iarrocomponent} to be generically étale locally  a product of elementary components. {  As already mentioned, when $X$ is smooth and $n$-dimensional, the elementary components of $\Hilb (X)$ are in canonical bijection with those of $\Hilb(\BA^n)$.}  Among them, the class of primitive ones is particularly important, as primitive elementary components do not arise from $\Hilb(\BA^{n'})$, for any $n'<n$.

\begin{remark}\label{rem:deftheta}
 Let us identify the partial derivatives $\partial_{x_j}$, for $j=1,\ldots,n$, with a basis of the tangent space $\mathsf T_0\BA^n$ and let us consider a fat nesting $[\underline{Z}]\in \hilbert{+}{\BA^n}$. In this setting, we have a homomorphism
  \[
    \begin{tikzcd}[row sep = tiny]
         \mathsf T_{0} \BA^n \arrow[r,"\widetilde\theta"] & \mathsf T_{[\underline{Z}]}\Hilb^{} (\BA^n)\\
         \partial_{x_j}\arrow[r,mapsto]&\left( \pi^{(i)}\circ \partial_{x_j}\right)_{i=1}^r,
    \end{tikzcd}
  \]
where $\pi^{(i)}:R\to R/I^{(i)}$, for $i=1,\ldots,r$, is the canonical projection. Here, we have used the canonical identification of the tangent space to the nested Hilbert scheme at $[\underline I]\in\Hilb{(\BA^n)}$ with the subspace of $\bigoplus_{i=1}^r\Hom_R(I^{(i)},R/I^{(i)})$ cut out by the nesting conditions, see \cite{sernesi} for more details.

  We denote by  $\theta : \mathsf T_{0} \BA^n \to \mathsf T_{ [\underline{Z}]}^{< 0} \Hilb^{} (\BA^n)
  $ the map defined as  the composition of $\widetilde \theta$ with the canonical projection defining the negative tangent space, see \Cref{def:negativetangents}.
\end{remark}
\begin{definition}\label{def:TNT}
    Let $[\underline{Z}]\in \Hilb^{+} (\BA^n)$ be a fat nesting. Then, $[\underline{Z}]$ has TNT (Trivial Negative Tangents) if the   map 
\[
         \mathsf T_{0} \BA^n  \xrightarrow{\theta} \mathsf T_{ [\underline{Z}]}^{< 0}  \Hilb(\BA^n)
\]
    is surjective.
\end{definition} 
\Cref{thm:tnt per nested} is a generalisation of {\cite[Theorem 4.9]{ELEMENTARY}} and it relates the existence of ideals having TNT and the existence of generically reduced elementary components. 
\begin{theorem}[{\cite[Theorem 4]{UPDATES}}] \label{thm:tnt per nested} 
Let $V\subset \Hilb^{} (\BA^n)$ be an irreducible component.  Suppose that $V$ is generically reduced. Then $V$ is elementary if and only if a general point of $V$ has TNT. 
\end{theorem}

\begin{remark}
    We will use the ideas presented in this subsection in two different ways. Indeed, in the proof of \Cref{THMINTRO:A} (\Cref{thm:mainAtext}) we first observe that some ideal has TNT thereby proving that it lies only on generically reduced elementary components. Then,  we compute the dimension of some Hilbert--Samuel strata to show that they do not fit in the smoothable component. In the proof of \Cref{THMINTRO:B} (\Cref{thm:mainB}) we show that the generic point on some elementary component does not have the TNT property, and this implies that the component is generically non-reduced. 
\end{remark}
 
\subsection{The class of 2-step ideals}\label{subsec 2-step}
In this subsection we briefly review the definition of a 2-step ideal  without linear syzygies from \cite{2step}, and we recall the exact formula for the dimension of the locus parametrising these objects.
\begin{definition}
A 2-step ideal $I\subset R$ of order $k\in\BZ_{>0}$ is an $\Fm$-primary ideal such that
\[
\mathfrak{m}^{k+2} \subset I \subset \mathfrak{m}^{k},\quad\mbox{ and }\quad I\not \subset \mathfrak{m}^{k+1}.
\] 
We say that the ideal $I$ is  without   linear syzygies if $\bh_I(k+1)-n\bh_I(k)\geqslant0$.
Similarly, a 2-step algebra is a local Artinian $\BC$-algebra of the form $R/I$ for a 2-step ideal $I$, and its Hilbert function is 2-step as well.
\end{definition}

\begin{prop}\label{prop:dim2step}
    Let $\bh$ be the Hilbert function of a 2-step ideal  without   linear syzygies, and let $\boldsymbol q = \bh_R-\bh $ be the Hilbert function of the corresponding quotient algebra. 
    Then, we have
    \[
    \dim H_{\boldsymbol q}^n =  \bh(k)\boldsymbol q(k) +\big(\bh({k+1}) -(n-1) \bh(k)\big)\boldsymbol q({k+1}) .
    \]
\end{prop}

\section{\texorpdfstring{Proof of \Cref{THMINTRO:A}}{Proof of \textbf{Theorem A}}}
\label{sec3}
In this section we prove \Cref{THMINTRO:A}. First we describe a collection of ideals having the Hilbert functions discussed in \Cref{prop:non-smoothable}, which in turn is a key tool in the proof of \Cref{THMINTRO:A}. We conjecture in \Cref{conj:TNT} that all the ideals belonging to this collection have TNT, and we confirm it in \Cref{prop:checked}, for $n\leqslant 15$. 

\begin{remark}\label{rem:gensI2}  
Fix an integer $n\geqslant 4$ and consider the ideals
\begin{equation}\label{eq:DeltaJ}
    \Delta_n=\rk\begin{pmatrix}
    x_1 & x_2 & \cdots &x_{n}\\
    x_n & x_1 & \cdots &x_{n-1} 
\end{pmatrix}\leqslant 1, \qquad
J_n = x_n(x_i+x_{n-1} \ |\ i=1,\ldots,n-2 ),
\end{equation}
and their sum $I^{(2)}=\Delta_n+J_n$. We omit the dependence on $n$ from $I^{(2)}$.

    We have $\bh_{R/\Delta_n}(i)=n$, for all $i\geqslant1$. A minimal set of generators for  $\Delta_n$ has the form
    \[
    \Set{x_ix_j-x_{i-1}x_{j+1}| 1\leqslant i\leqslant j < n },
    \]
     with the convention $x_{0}=x_n$. Moreover, we have $\bh_{R/I^{(2)}}=(1,n,2)$.
\end{remark}

\begin{example}\label{ex:dim4}This example first appeared in \cite{2step}. For $n=4$, the ideal
    \[
    I^{(2)}=\Delta_4+J_4 ,
    \]
    has minimal free resolution of the form
    \[
    0\leftarrow I^{(2)}\leftarrow R(-2)^{\oplus8}\leftarrow R(-3)^{\oplus 12} \oplus R(-4)\leftarrow R(-4)^{\oplus 4} \oplus R(-5)^{\oplus 4}\leftarrow R(-6)^{\oplus 2}  \leftarrow 0.
    \]
    A direct check shows that, denoted by $I^{(1)}$ the ideal $I^{(1)}=(x_1,x_2)^2+(x_3,x_4)$, we have
    \[
    \mathsf t^{=-1}_{[I^{(1)}\supset I^{(2)}]}\Hilb(\BA^4)=4.
    \]
    In other words, the point $[I^{(1)}\supset I^{(2)}]\in\Hilb(\BA^4)$ has TNT, and hence it lies on a generically reduced elementary component, see \cite{2step} for the details.
    Both the resolution and the dimension of the tangent space can be computed via the following \textit{Macaulay2} code using the package \cite{HilbQuotPaoloLella}.
    \begin{verbatim}
    loadPackage "HilbertAndQuotSchemesOfPoints";
    n=4;     -- n >= 4
    R=QQ[x_1..x_n];
    Delta = minors(2,matrix {{x_1 .. x_n},
                             {x_n} | {x_1 .. x_(n-1)}});
    J = x_n*ideal( for i from 1 to n-2 list (x_i+x_(n-1)) );
    I2 = Delta + J;
    betti res I2
    s = 2;     -- 2 <= s <= n-2
    I1 = (ideal (for i from 1 to s list x_i))^2 +
          ideal (for i from s+1 to n list x_i);
    point = nestedHilbertSchemePoint {I1,I2};
    TS = tangentSpace point;
    hilbertSeries TS
    hilbertFunction(-1,TS)
    \end{verbatim}
\end{example}

We show now that the behaviour in \Cref{ex:dim4} can be promoted to a general statement. Precisely, we prove that, for $n\geqslant 4$, the nestings $I^{(1)}\supset I^{(2)}$, constructed from \eqref{eq:DeltaJ}, are non-smoothable. We expect them to have TNT as we conjecture in \Cref{conj:TNT}, see also \Cref{rem:expect} and \Cref{prop:checked}.
 
{ 
\begin{prop}\label{prop:non-smoothable}
Fix a positive integer $n\geqslant4$. Then, the generic nesting with vector of Hilbert functions 
    \[
    \underline{\bh}=((1,s),(1,k,2)),
    \]
    for $4\leqslant k \leqslant n$, and $2\leqslant s\leqslant k-2$ is non-smoothable. Moreover the stratum $H^n_{\underline{\bh}}$ is irreducible.
\end{prop}
\begin{proof} Thanks to \Cref{lemma:tech2}, it is enough to consider the case $k=n$, where the dimension of the smoothable component is $n(n+3)$. The stratum $H_{(1,n,2)}^n$ parametrises  so-called very compressed algebras. As a consequence, it agrees with the homogeneous locus and it is isomorphic to a Grassmannian $\Gr(2,R_2)=\Gr\left(2,\binom{n+1}{2}\right)$, see \cite[Lemma 5.1.]{JOACHIM}. On the other hand, the morphism
 \[
 \begin{tikzcd}
  H^n_{\underline{\bh}}\arrow[r]&H^n_{(1,n,2)},   
 \end{tikzcd}
 \]
defines a trivial $\Gr(s,R_1)$-bundle. Therefore, the stratum  $H^n_{\underline{\bh}}$ has dimension 
 \[
 \dim H^n_{\underline{\bh}}=n^2+n -4 +s(n-s),
 \]
 and is irreducible. Now we compare the dimensions of the smoothable component $\Hilb^{s+1,n+3}_{\sm}(\BA^n)$ and of the locally closed subset $H^n_{\underline{\bh}}\times\BA^n\subset\Hilb^{s+1,n+3}(\BA^n)$. We have
\begin{equation}\label{eq:gap}
    \dim \Hilb^{s+1,n+3}_{\sm}(\BA^n)- \dim H^n_{\underline{\bh}}\times\BA^n = n +4-s(n-s)= n(1-s) +4+s^2,
\end{equation}
which, for $n\geqslant 8$, is non-positive if and only if  $2\leqslant s \leqslant n-2$. In particular, in this range the generic $\underline{\bh}$-nesting is non-smoothable. For $4\leqslant n \leqslant 8$, a direct check shows that the generic $\underline{\bh}$-nesting has TNT, see \Cref{ex:dim4} and \Cref{prop:checked}.    
\end{proof}

\begin{remark}
 As anticipated in the introduction, it was shown in \cite{ERMANVELASCO} that the smallest integer $e$ for which a fat point having Hilbert function $(1,n,e)$ is non-smoothable is $e=3$. In contrast, \Cref{prop:non-smoothable} shows that, by allowing the presence of nestings, this minimum becomes smaller.
\end{remark}
\begin{lemma}\label{lemm:piccsmoo}
    Fix integers $ n\geqslant 4$ and $0\leqslant s \leqslant n$, and consider the following nesting of Hilbert functions
    \[
    \underline{\bh}_{1}=((1,s),(1,n)),\quad \underline{\bh}_{2}=((1,s),(1,n,1)). 
    \]
    Then, the corresponding Hilbert--Samuel strata are contained in the smoothable component.
\end{lemma}
\begin{proof}
We start by showing that $\underline{\bh}_{1}$-nestings are smoothable. Let $[I_1\supset J_1] \in H_{\underline{\bh}_{1}}$ be a closed point. Without loss of generality, we can assume $I_1=(x_{s+1},\ldots,x_n)+\Fm^2$ and $J_1 = \Fm^2  $. {These are monomial ideals and the conclusion is an application of the theory of canonical distractions. More precisely, one applies \cite[Theorem 4.9]{HartshorneConnect} verbatim to the nested setting.}

We move to the second case. For a generic point $[I_2\supset J_2] \in H_{\underline{\bh}_{2}}$, we may assume that $R/J_2$ is Gorenstein. In other words, we can put $ I_2=I_1$ and $J_2= (x_ix_j\ |\ 1\leqslant i< j\leqslant n) + (x_1^2-x_i^2\ |\  i =2,\ldots,n)$. Then, the families 
\[
\widetilde{I}_2=(x_1^2,x_{s+1},\ldots,x_n)+ (x_ix_j\ |\ 1\leqslant i< j\leqslant s) + (x_i(x_i-t)\ | \ i=2,\ldots,s ) \subset R[t],
\]
\[\widetilde{J}_2= (x_ix_j\ |\ 1\leqslant i< j\leqslant n) + (x_1^2-x_i(x_i - t)\ | \ i=2,\ldots,n )\subset R[t],
\]
    contain points corresponding to disjoint unions of curvilinear schemes and distinct singletons, hence smooth points of the smoothable component. Indeed, one can immediately check that $\widetilde{J}_2=K_1\cap\cdots \cap K_n  $, where 
\[
K_1= (x_n^2,x_1 x_n,x_1^2+x_n t)+ (x_i-x_n\ |\ i=2,\ldots,n-1), 
\]
and $K_i = (x_1,\ldots,x_{i-1},x_i-t,x_{i+1},\ldots,x_n)$, for $i=2,\ldots,n$. Similarly for $\widetilde{I}_2$. This proves the second part of the statement. 
\end{proof}

\begin{theorem}[\Cref{THMINTRO:A}]\label{thm:mainAtext}
   Let $R$ be the polynomial ring in $n\geqslant 4$ variables and let $2\leqslant s \leqslant n-2$ be an integer. Then, the scheme $ \Hilb^{1+s,n+3}(\BA^n)$ has a primitive elementary component.  
\end{theorem}
\begin{proof}
    Let us consider the following sequence of Hilbert functions
    \[
    \underline{\bh}=((1,s),(1,n,2)).
    \]
    The generic point $[I^{(1)}\supset I^{(2)}]\in H_{\underline{\bh}}$ has the form 
   \[
   I^{(1)}=(x_1,\ldots,x_s)^2+(x_{s+1},\ldots,x_n),\mbox{ and } I^{(2)}=(H_2)+\Fm^3,
   \]
    for some generic linear subspace  $H_2\subset R_2$  of codimension $2$ of the space of quadratic forms.  Then, the point $[I^{(1)}\supset I^{(2)}]\in\Hilb^{1+s,n+3}(\BA^n)$ is non-smoothable. Indeed, this is an open property and the generic nesting in a  (component of a) stratum is non-smoothable if the stratum contains a non-smoothable point. For example, the ideal $I^{(2)}=\Delta_n+J_n$, from \eqref{eq:DeltaJ}, defines a non-smoothable point. Indeed, it is generated by $\left(\binom{n+1}{2}-2\right)$ quadrics and we have $\bh_{R/I^{(2)}}=(1,n,2)$, see \Cref{rem:gensI2}. Similarly, we have $\bh_{R/I^{(1)}}=(1,s)$. Therefore, we can apply \Cref{prop:non-smoothable} and we conclude that the nesting in non-smoothable. Note also that the socles of the algebras $R/I^{(i)}$, for $i=1,2$, are concentrated in maximal degree by the genericity assumption. Suppose by contradiction that the nesting $I^{(1)}\supset I^{(2)} $ does not lie on an elementary component. 
    Then, there are  families $\mathscr Z_1\subset\mathscr Z_2\subset B\times \BA^n$ parametrised by an irreducible pointed scheme $(B\ni b_0)$ such that the fibres $(\mathscr Z_i)_{b_0} = \mathscr Z_i\cap (\Set{b_0}\times \BA^n)$, for $i=1,2$, are the schemes defined by $I^{(1)}\supset I^{(2)}$ and the generic fibres do not correspond to nestings of fat points. Then, because of the observation on the socles, the generic fibre of $\mathscr Z$ has the form 
    \[
    Z_b=(X_1\subset X_2)\amalg (Y_1\subset Y_2),
    \]
    where both $(X_1\subset X_2)$ and $(Y_1\subset Y_2)$ are of one of the types considered in \Cref{lemm:piccsmoo}. This, together with \Cref{lemm:piccsmoo}, concludes the proof, since smoothability is a closed property.
\end{proof}
\begin{remark}\label{rem:expect} 
It is remarkable that, in all cases we have checked, the  nestings in \Cref{thm:mainAtext} have the TNT property. In particular, this provides a generically reduced elementary component of $\Hilb(\BA^n)$. Our check consists in showing that the ideal $I^{(2)}=\Delta_n+J_n$, from \eqref{eq:DeltaJ}, has the TNT property for a given $n\geqslant 4$. This is enough because the TNT property is an open condition, see \cite{ELEMENTARY}. This evidence, together with the apparent combinatorial structure of the considered ideals, leads us to formulate \Cref{conj:TNT}, which we are able to confirm up to dimension $n=15$ in \Cref{prop:checked}.
\end{remark}
\begin{conjecture}\label{conj:TNT}
   Fix an integer $n\geqslant 4$ and consider a nesting
    \[
    \underline{\bh}=((1,s),(1,k,2)).
    \]
    Then, $V_{\underline{\bh}}=\overline{H^n_{\underline{\bh}}}\times \BA^n$ is a generically reduced   elementary component if and only if  $
        4\leqslant k\leqslant n $, and $ 2\leqslant s\leqslant k-2 $. 
\end{conjecture}
\begin{prop}\label{prop:checked}
    \Cref{conj:TNT} is true for $n\leqslant 15$.
\end{prop}
\begin{proof}
By openness of the TNT property, see \Cref{rem:expect}, in order to prove the statement it is enough to exhibit a $\underline{\bh}$-nesting having TNT. We have almost no choice for $I^{(1)}$ and we choose $I^{(2)}=\Delta_n+J_n$, from \eqref{eq:DeltaJ}. Now, the proof is a direct check using the code in \Cref{ex:dim4}.
\end{proof}

 }

\section{\texorpdfstring{Proof of \Cref{THMINTRO:B}}{Proof of \textbf{Theorem B}}}
\label{sec4}
This section is devoted to the proof of \Cref{THMINTRO:B} from the introduction. First we prove a technical lemma relating the tangent space at a point $[\underline{I}]\in\Hilb(\BA^n)$ having TNT and the tangent space at a point obtained by sandwiching some power of the maximal ideal $\Fm$ in the nesting $\underline{I}$, and then we prove our main theorem. Along the way, we prove a stronger result only involving homogeneous ideals.
{ 
\begin{lemma}\label{lemma:tech3}
Let $[\underline{I}]\in \Hilb^+(\BA^n)$ be a point having TNT.  Assume that there is a  $j\in\Set{1,\ldots,r}$ such that 
    \[
    I^{(j)} \supsetneq \Fm^k \supsetneq I^{(j+1)},
    \]
    for some $k>0$, where we assume $I^{(r+1)}=0$ by convention. Consider the point
    \[
 p  = [I^{(1)}\supset\cdots \supset I^{(j)}\supset\Fm^k\supset I^{(j+1)}\supset\cdots\supset I^{(r)}]\in\Hilb(\BA^n).
    \]
Then, there is an isomorphism
    \begin{equation}\label{eq:isolem41}
         \mathsf{T}_{p}^{<0}\Hilb(\BA^n) \cong \mathsf{T}_{[\underline{I}]}^{<0}\Hilb(\BA^n) \oplus  
        \Hom_{R}  \left( \Fm^k/I^{(j+1)},I^{(j)}/\Fm^k \right) . 
    \end{equation}  
\end{lemma}
\begin{proof} Recall that 
  the TNT property implies that the tangent space $\mathsf T_{[{\underline{I}}]}\Hilb(\BA^n)$ decomposes as 
  \[
  \mathsf T _{[{\underline{I}}]}\Hilb(\BA^n) \cong \left\langle \partial_{x_i}^{\underline{I}} \ \big|\ i=1,\ldots,n \right\rangle_{\BC}  \oplus \mathsf T^{\geqslant 0}_{[{\underline{I}}]}\Hilb(\BA^n),
  \]
  where $\partial_{x_i}^{\underline{I}}$ denotes the sequence $(\pi^{(i)}\circ\partial_{x_j})_{i=1}^r$, see \Cref{rem:deftheta}. In particular,  we have $\mathsf T ^{<0}_{[{\underline{I}}]}\Hilb(\BA^n) \cong \left\langle \partial_{x_i}^{\underline{I}} \ \big|\ i=1,\ldots,n \right\rangle_{\BC}$. Although the negative tangent space is a quotient of the whole tangent space, with abuse of notation we will make extensive use of this isomorphism and we will interpret  $\mathsf T ^{<0}_{[{\underline{I}}]}\Hilb(\BA^n) $ as a subspace of $\mathsf T _{[{\underline{I}}]}\Hilb(\BA^n) $.
  
Recall also that 
\begin{equation}\label{eq:tgmk}
    \mathsf T_{[\Fm^k]}\Hilb(\BA^n)\cong\mathsf T_{[\Fm^k]}^{<0}\Hilb(\BA^n) \cong  \Hom_R(\Fm^k,R/\Fm^k)_{-1}\cong\Hom_{\BC}(R_k,R_{k-1}) .
\end{equation}
Denote by $\mathsf D  $ the subspace
\[
\mathsf D= \left\langle{ (\partial_{x_i}^{\underline{I}},\partial_{x_i}^{\Fm^k})\ \big|\ i=1,\ldots,n }\right\rangle_{\BC} \subset \mathsf{T}_{p}^{<0}\Hilb(\BA^n) \subset \left\langle \partial_{x_i}^{\underline{I}} \ \big|\ i=1,\ldots,n \right\rangle_{\BC}\oplus \Hom_R(\Fm^k,\Fm^{k-1}/\Fm^k).
\]
{  By \eqref{eq:tgmk}}, we have $\mathsf D\cong \mathsf T_{[\underline{I}]}^{<0}\Hilb(\BA^n)$. 
 
 We look for a complement of $\mathsf D$ realising the isomorphism \eqref{eq:isolem41}, i.e.,~a vector subspace $\mathsf V\subset  \mathsf{T}_{p}^{<0}\Hilb(\BA^n) $ such that  
 \begin{equation}
     \label{eq:isoV} \mathsf V\cong \Hom_{R}  \left( \Fm^k/I^{(j+1)},I^{(j)}/\Fm^k \right),
 \end{equation}
 and $\mathsf{T}_{p}^{<0}\Hilb(\BA^n)\cong \mathsf D\oplus\mathsf V$. Note that, one possible such $\mathsf V$ can be identified with the subspace of $\mathsf T_{[\Fm^k]}\Hilb(\BA^n) $ consisting of tangents $\varphi_k$making the following diagram 
    \begin{equation}
        \label{eq:diaginc}
        \begin{tikzcd}
    I^{(j)}\arrow[d,"0"']& \Fm^k \arrow[l,hook',"\iota_j"']\arrow[d,"\varphi_k"]& I^{(j+1)}\arrow[l,hook',"\iota_{j+1}"']\arrow[d,"0"]\\
    R/I^{(j)}& R/\Fm^k \arrow[l,two heads,"\pi_j"]& R/I^{(j+1)}\arrow[l,two heads,"\pi_{j+1}"]
\end{tikzcd}
    \end{equation}
commute. 

First, the condition $\pi_j\circ \varphi_k\equiv 0 $ implies that $\varphi_k(\Fm^k)\subset I^{(j)}/\Fm^k$. On the other hand, since $I^{(j+1)}\subset \ker \varphi_k$ the morphism $\varphi_k$ factors uniquely through the kernel $\Fm^k/I^{(j+1)}$ of $\pi_{j+1}$. This gives an isomorphism \eqref{eq:isoV}, and hence the decomposition in \eqref{eq:isolem41}.  
\end{proof}
}

In the homogeneous setting, we can prove {  the following} much stronger result.
\begin{prop}\label{cor:thmB-2step-homogeneous}
   Let   $[I^{(1)} \supset I^{(2)}]\in (\Hilb(\BA^n))^{\BG_m}$ be a nesting of homogeneous  ideals    such that $I^{(1)}\supset \Fm^k\supset I^{(2)} $ and 
   \[
   \mathsf T_{[I^{(1)}  \supset I^{(2)}]} ^{< 0}\Hilb(\BA^n) \cong \mathsf T_{[I^{(1)} \supset I^{(2)}]} ^{= -1}\Hilb(\BA^n).
   \]
      Then, we have 
      \begin{itemize}
        \item $\mathsf T_{[I^{(1)} \supset \mathfrak{m}^k \supset I^{(2)}]} ^{\geqslant 0}\Hilb(\BA^n) \cong \mathsf T_{[I^{(1)}\supset I^{(2)}]} ^{\geqslant 0}\Hilb(\BA^n) $,
        \item $\mathsf T_{[I^{(1)} \supset \mathfrak{m}^k \supset I^{(2)}]} ^{< 0}\Hilb(\BA^n) \cong \mathsf T_{[I^{(1)} \supset \mathfrak{m}^k \supset I^{(2)}]} ^{= -1}\Hilb(\BA^n)$ \\ \hspace*{3.05cm}${}\cong \mathsf T_{[I^{(1)}   \supset I^{(2)}]} ^{= -1}\Hilb(\BA^n) \oplus \Hom_{\BC}\big((R/I^{(2)})_{k},I^{(1)}_{k-1}\big)$.
        \end{itemize}
\end{prop}
\begin{proof}    The first isomorphism is a consequence of the vanishing $\mathsf t^{\geqslant 0}_{[\Fm^k]}\Hilb(\BA^n)= 0$ and of the isomorphism  $\mathsf T^{\geqslant 0}_{[I^{(1)}\supset I^{(2)}]}\Hilb(\BA^n)\cong \mathsf T _{[I^{(1)}\supset I^{(2)}]}\Hilb^+(\BA^n)$, see \Cref{rem:nonneg-punctual} and \Cref{rem:BBHOMOo}.  The second isomorphism is a consequence of our assumptions on the nesting $I^{(1)}\supset I^{(2)}$, indeed both $[I^{(1)}\supset I^{(2)}]$ and $[\Fm^k]$ have negative tangent spaces concentrated in degree -1.
     We move now to the proof of the last isomorphism. We start by showing that the natural projection 
     \[
\begin{tikzpicture}
         \node (A) at (0,0) [] {$\mathsf T_{[I^{(1)} \supset \mathfrak{m}^k \supset I^{(2)}]} ^{=-1}\Hilb(\BA^n)$};
         \node (B) at (3.75,0) [] {$\mathsf T_{[I^{(1)} \supset   I^{(2)}]} ^{=-1}\Hilb(\BA^n)$};
         \draw [->] (A) --node[above]{\tiny $\pi$} (B);
\end{tikzpicture}
     \]
     is surjective.
     We look at the $-1$-tangent vectors of the triple $[I^{(1)} \supset \mathfrak{m}^k \supset I^{(2)}]$, i.e.,~triples of tangent vectors making the diagram \eqref{eq:diaginc} commute. Because of the   assumption on the negative tangent space and of the isomorphisms in \eqref{eq:tgmk}, we can focus on the homogeneous  pieces of degrees $k$ and $k-1$ of the ideals and of the quotients respectively. By the assumption on the nesting, at the level of vector spaces, we have
\begin{itemize}
\item $I_k^{(1)} = (\mathfrak{m}^k)_k = R_k \cong I^{(2)}_k \oplus (R/I^{(2)})_k$,
\item $(R/I^{(2)})_{k-1} = (R/\mathfrak{m}^k)_{k-1} = R_{k-1} \cong (R/I^{(1)})_{k-1} \oplus I^{(1)}_{k-1}$.
\end{itemize}

Surjectivity of $\pi$ is equivalent to the fact that, if $[\varphi^{(1)},\varphi^{(2)}]\in\mathsf T^{=-1}_{[I^{(1)}\supset I^{(2)}]}\Hilb(\BA^n)$ is a tangent vector, then there exists a $\varphi^k\in\mathsf T^{=-1}_{[\Fm^k]}\Hilb(\BA^n)$ making the following diagram commute.
\begin{equation}\label{eq:vector-space-level2}
\begin{tikzpicture}[xscale=2]
\node at (0,0.025) [] {$\overbrace{I^{(2)}_k \oplus (R/I^{(2)})_k}^{(\mathfrak{m}^k)_k}$};
\node at (-2,0) [] {$\overbrace{I^{(2)}_k \oplus (R/I^{(2)})_k}^{I^{(1)}_k}$};
\node at (2,-0.25) [] {$I^{(2)}_k$};

\node at (-2,-2) [] {$(R/I^{(1)})_{k-1}$};
\node at (0,-2.27) [] {$\underbrace{(R/I^{(1)})_{k-1} \oplus I^{(1)}_{k-1}}_{(R/\mathfrak{m}^k)_{k-1}}$};
\node at (2,-2.27) [] {$\underbrace{(R/I^{(1)})_{k-1} \oplus I^{(1)}_{k-1}.}_{(R/I^{(2)})_{k-1}}$};

\draw [-latex',dotted] (0,-0.5) --node[right]{\tiny $\varphi_k$} (0,-1.75);
\draw [-latex'] (-2,-0.5) --node[left]{\tiny $\left[\begin{array}{c|c} \varphi^{(1)}_1& \varphi_2^{(1)} \end{array}\right]$} (-2,-1.75);
\draw [-latex'] (2,-0.5) --node[right]{\tiny $\left[\begin{array}{c} \varphi^{(2)}_1 \\ \hline \varphi^{(2)}_2\end{array}\right]$} (2,-1.75);

\draw [<-left hook] (-1.5,-0.25) --node[above]{\tiny $\left[\begin{array}{c|c} \text{id}& 0 \\ \hline 0&\text{id}\end{array}\right]$} (-0.5,-0.25);
\draw [<-left hook] (0.5,-0.25) --node[above]{\tiny $\left[\begin{array}{c} \text{id} \\ \hline 0\end{array}\right]$} (1.85,-0.25);

\draw [<<-] (-1.65,-2.) --node[above]{\tiny $\left[\begin{array}{c|c} \text{id}& 0 \end{array}\right]$} (-0.65,-2.);
\draw [<<-] (0.65,-2.) --node[above]{\tiny $\left[\begin{array}{c|c} \text{id}& 0 \\ \hline 0&\text{id} \end{array}\right]$} (1.35,-2.);
\end{tikzpicture}
\end{equation}
In particular, we have $\varphi_1^{(1)}\equiv \varphi_1^{(2)}$. A possible solution is 
\[
\varphi_k = \left[\begin{array}{c|c} \varphi^{(1)}_1& \varphi_2^{(1)} \\ \hline \varphi_2^{(2)}& 0 \end{array}\right].
\]
Note that the isomorphisms in \eqref{eq:tgmk} ensure that $\varphi_k$ defines uniquely an element in $\Hom_R(\Fm^k,R/\Fm^k)$. Note also that this construction defines a section 
\[
\begin{tikzcd}
    \mathsf T_{[I^{(1)} \supset   I^{(2)}]} ^{=-1}\Hilb(\BA^n)\arrow[r,"\sigma",hook] &     \mathsf T_{[I^{(1)} \supset \mathfrak{m}^k \supset I^{(2)}]} ^{=-1}\Hilb(\BA^n)
\end{tikzcd} 
\]
of $\pi$. In order to conclude, we need to compute a complement of $\sigma\left( \mathsf T_{[I^{(1)}\supset I^{(2)}]} ^{=-1}\Hilb(\BA^n)\right)$ as we did in \Cref{lemma:tech3}. This is equivalent to finding the tangents $\varphi_k\in\mathsf T^{=-1}_{[\Fm^k]}\Hilb(\BA^n)$ making the diagram \eqref{eq:vector-space-level2} commute with $\varphi_i^{(j)}=0$, for $i,j=1,2$. This is the case if 
\[
\varphi_k = \left[\begin{array}{c|c} 0 & 0 \\ \hline 0& \psi_k \end{array}\right],
\]
for {  every} $\psi_k \in \Hom_{\BC}\big((R/I^{(2)})_{k},I^{(1)}_{k-1}\big)$. This concludes the proof.
\end{proof}

The following example shows that the hypotheses of \Cref{cor:thmB-2step-homogeneous} are sharp.
\begin{example}
    Consider the nesting of ideals in 4 variables
    \[
    I^{(1)}=\Fm^2\supset\Fm^3\supset I^{(2)}=x_4\cdot\Fm^2+ (x_3x_1,x_3x_2,x_2^2)\cdot\Fm+ (x_1^4,x_3^5),
    \]
    with corresponding Hilbert functions
    \[
    \bh_{R/I^{(1)}}=(1,4 ),\quad \bh_{R/I^{(2)}}=(1, 4, 10, 3, 2 ).
    \]
    Then, a direct computation using  the package \cite{HilbQuotPaoloLella} shows that 
    \[
    \mathsf t_{[I^{(1)}\supset I^{(2)}]}^{=-2}\Hilb(\BA^4)=10,\qquad
        \mathsf t_{[I^{(1)}\supset I^{(2)}]}^{=-3}\Hilb(\BA^4)=8,\qquad
            \mathsf t_{[I^{(1)}\supset I^{(2)}]}^{=-4}\Hilb(\BA^4)=0,
    \]
    and  
    \[
    \mathsf t^{<0}_{[I^{(1)}\supset \Fm^3\supset I^{(2)}]}\Hilb(\BA^4) = 
    \mathsf t^{=-1}_{[I^{(1)}\supset I^{(2)}]}\Hilb(\BA^4) + \hom_{\BC}((R/I^{(2)})_2,I^{(1)}_{2})+1.   
    \]
\end{example}

We are now in a position to prove \Cref{THMINTRO:B} from the introduction.
\begin{theorem}[\Cref{THMINTRO:B}]\label{thm:mainB}
     Let $V\subset \Hilb^{\underline{d}} (\BA^n)$ be a generically reduced elementary component. Assume that there is at least one point $[\underline{I}]\in V$ and some $k\in\BZ_{>0}$ such that $I^{(j)}\supsetneq \Fm^k\supsetneq I^{(j+1)} $, for some $1\leqslant j\leqslant r-1$. Put 
    \[
    \widetilde{\underline {d}}=\left(d_1,\ldots, d_j,\binom{n+k-1}{n},d_{j+1},\ldots,d_r\right).
    \]  
    Then, there is a generically non-reduced elementary component $\widetilde V\subset \Hilb^{\widetilde{\underline{d}}}(\BA^n)$ such that  
$ \widetilde V_{\red}\cong   V_{\red} . $
\end{theorem} 

\begin{proof}
 Let $\bh^{k}$ be the Hilbert function of the local algebra $R/\Fm^k$. 
    Let also $\underline{\bh}$ be the $r$-tuple of Hilbert functions such that the closure  of $H^n_{\underline{\bh}}\times\BA^n$ in $\Hilb(\BA^n)$ is the elementary component $V$. Denote by $\widetilde{\underline{\bh}}$ the sequence 
    \[
    \widetilde{\underline{\bh}}={(\bh_1,\ldots,\bh_j,\bh^k,\bh_{j+1},\ldots,\bh_r)}. 
    \]
    Then, by hypothesis, the power $\Fm^k$ of the maximal ideal can be strictly sandwiched in between the $j$-th and $(j+1)$-st ideals of  every nesting in $V$. Moreover,  the closure $\widetilde V$ of $H^n_{\widetilde{\underline{\bh}}}\times\BA^n$ in $\Hilb(\BA^n)$ is an elementary component. 
    Consider the forgetful morphism 
    \[
    \begin{tikzcd} H^n_{\widetilde{\underline{\bh}}}\arrow[r]&H^n_{\underline{\bh}}.
    \end{tikzcd}
    \]
    This is an isomorphism. Indeed, it is clearly bijective and, at the level of tangent spaces, it is an isomorphism by the results in \cite[Corollary 3.10]{2step}, see \Cref{subsec:HS}. On the other hand, the tangent space to the whole Hilbert scheme also consists of negative tangents. The dimensions of the respective negative parts of the tangent spaces do not agree as it can be computed via \Cref{lemma:tech3}. This concludes the proof.
\end{proof}

\begin{remark}\label{rem:generalisationold}
    Note that \Cref{thm:mainB} recovers \cite[Theorem 5]{UPDATES}. For instance, if $[I]\in\Hilb(\BA^n)$  corresponds to an $\Fm$-primary ideal having TNT, we can consider the commutative diagram
    \[
    \begin{tikzcd}
    R\arrow[d ]& \Fm   \arrow[l,hook' ]\arrow[d ]& I \arrow[l,hook' ]\arrow[d ]\\
    0& R/\Fm \arrow[l,two heads ]& R/I \arrow[l,two heads ]
\end{tikzcd}
    \]
    and  apply \Cref{thm:mainB}. This gives
    \[
      \mathsf t^{=-1}_{[\Fm\supset I]}\Hilb(\BA^n)=     \mathsf t^{=-1}_{[ I]}\Hilb(\BA^n)+   \hom_R(\Fm/I,R/\Fm)=n+\bh_{R/I}(1).
    \]
\end{remark}

\begin{example}
Fix some $4\leqslant n\leqslant  15$ and $2 \leqslant s \leqslant n-2$. Consider a point $[I^{(1)} \supset I^{(2)}] \in \Hilb^{s+1,n+3}(\mathbb{A}^n)$ with the TNT property, see \Cref{prop:checked}.   Hence, by  \Cref{cor:thmB-2step-homogeneous}, the Hilbert scheme $\Hilb^{s+1,n+1,n+3} (\mathbb{A}^n)$ has a generically non-reduced component whose dimension is $n(n-s) + 2\left(\tbinom{n+1}{2}-2\right) + n$ and
\[
\mathsf t^{=-1}_{[I^{(1)}  \supset \mathfrak{m}^2 \supset I^{(2)}]}\Hilb^{s+1,n+1,n+3}(\BA^n) = n +  \bh_{R/I^{(2)}} (2)   \bh_{I^{(1)}} (1) = n + 2(n-s).
\] 
For general $n$, we did not prove TNT. However, the negative tangent space is concentrated in degree $-1$ and we get
\[
\mathsf t^{=-1}_{[I^{(1)}  \supset \mathfrak{m}^2 \supset I^{(2)}]}\Hilb^{s+1,n+1,n+3}(\BA^n) = \mathsf t^{=-1}_{[I^{(1)}    \supset I^{(2)}]}\Hilb^{s+1,n+3}(\BA^n) + 2(n-s).
\] 
\end{example}

\begin{example}
By direct computation via Macaulay2 one can show that there exists a point $[I^{(1)}\supset I^{(2)}]\in (\Hilb (\BA^4))^{\BG_m}$ with  Hilbert functions
\[
\bh_{R/I^{(1)}} = (1,4,3)\qquad\text{and}\qquad \bh_{R/I^{(2)}} = (1,4,10,18,10)
\]
having the TNT property. We have
\[
 I^{(1)} \supsetneq \mathfrak{m}^3 \supsetneq I^{(2)} ,
\]
as required in \Cref{cor:thmB-2step-homogeneous}. The dimension of the tangent space to $\Hilb^{8,15,43} (\mathbb{A}^4)$ at the point $[I^{(1)} \supset \Fm^3 \supset I^{(2)}]$ jumps by
\[
\hom_{\BC}\big(I^{(1)}_2,(R/I^{(2)})_3\big) = \bh_{I^{(1)}} (2)  \bh_{ R/I^{(2)}}(3) = 7 \cdot 18 = 126.
\]
\end{example}

\section{\texorpdfstring{Proof of \Cref{THMINTRO:C}}{Proof of \textbf{Theorem C}}}
\label{sec5}
In this section we study the Hilbert scheme of points on singular hypersurfaces of $\BA^3$, and we prove \Cref{THMINTRO:C} from the introduction. We start with two examples. The first focuses on surface singularities of embedding dimension higher than 3. The second is an application of a result by Iarrobino from \cite{IARRO}. We emphasize that neither example is entirely new, as both can be deduced from \cite{Iarrob} and \cite{IARRO} using simple arguments. However, we present them here because they serve as motivating examples for our main result.

\begin{example}\label{ex:RNC}
    Consider the surface $S\subset \BA^4$ defined by the following ideal
    \[
    I_S=\rk\begin{pmatrix}
        x_1&x_2&x_3\\x_2&x_3&x_4
    \end{pmatrix}\leqslant 1.
    \]
    Note that $S$ is an affine chart of the vertex of a cone over the twisted cubic. Let now $Z\subset \BA^4$ be the zero-dimensional subscheme of length 8 defined by
    \[
    I_Z= (x_1,x_3)^2+(x_2,x_4)^2+  \det \begin{pmatrix}
        x_1 &x_3\\x_2 &x_4
    \end{pmatrix}.
    \]
    Recall that $[Z]\in\Hilb^8(\BA^4) $ has the TNT property and hence $Z$ is non-smoothable; see \cite[Proposition 5.1.]{8POINTS} or \cite[\S 2.2]{Iarrob}. Clearly, we have $I_S\subset I_Z$ and hence $Z\subset S$. As a consequence, the scheme $\Hilb^8 (S)$ is reducible.
\end{example}
\begin{example}\label{ex:iarrosurf}
    As shown in \cite[Example 4.3]{IARRO}, the generic local Artinian algebra $(A,\Fm_A)$  having Hilbert function $(1,3,6,10,15,21,17,5)$ is non-smoothable. Note that $\Fm_A^8\equiv 0$. As a consequence, for every hypersurface $S\subset \BA^3$ having a singularity of multiplicity at least 8, the scheme $\Hilb^{78}(S)$ is reducible.
\end{example}

\begin{remark}
    In \cite{IARRO}, the author argues via dimension count. Precisely, he estimates the dimension of the locus parametrising compressed algebras of given socle type, showing the existence of loci too large to lie in the smoothable component. In \cite{2step} the authors also apply a similar strategy  to a different class of algebras, namely 2-step ones. Here, we apply the same techniques to 2-step algebras corresponding to zero-dimensional closed subschemes of hypersurfaces of $\BA^3$.
\end{remark}

\begin{theorem}[\Cref{THMINTRO:C}]\label{thm:mainC}
Let $S\subset \BA^3$ be a hypersurface. If $S$ has a singular point of multiplicity at least 5, then $\Hilb^d (S) $ is reducible for $d\geqslant 22$.
 \end{theorem}
\begin{proof}
Let us denote by $f_S\in R$ the equation of $S$. Recall also that the closed immersion $S\hookrightarrow \BA^3$ induces a closed immersion $\Hilb (S)\hookrightarrow \Hilb (\BA^3)$. Consider the 2-step Hilbert function $\bh=(1,3,6,8,4)$ of size $|\bh|=22$. Note that $\bh$ has length 5. Therefore, for every ideal $I\subset R$ such that $\bh_{R/I}\equiv \bh$, we have $f_S\in I$.  As a consequence, we have $H_{\bh}^3\subset \Hilb^{22} (S)$. To conclude, note that the smoothable component  $\Hilb^{22}_{\sm}( S)$ has dimension 44, while we have $\dim H^3_{\bh}=44$, by \Cref{prop:dim2step}.
\end{proof}

\providecommand{\bysame}{\leavevmode\hbox to3em{\hrulefill}\thinspace}
\providecommand{\MR}{\relax\ifhmode\unskip\space\fi MR }
\providecommand{\MRhref}[2]{%
  \href{http://www.ams.org/mathscinet-getitem?mr=#1}{#2}
}

\end{document}